\documentclass{article}
\usepackage{amssymb}
\usepackage{amsthm}
\usepackage{amsmath}

{\theoremstyle{plain}%
  \newtheorem{theorem}{Theorem}[section]
  
  \newtheorem{proposition}{Proposition}[section]
  \newtheorem{lemma}{Lemma}[section]
  \newtheorem{note}{Note}[section]
  \newtheorem{Definition}{Definition}[section]

{\theoremstyle{remark}

\newtheorem{remark}{Remark}[section]
}
{\theoremstyle{definition}

\newtheorem{example}{Example}[section]
}
\newfont{\hueca}{msbm10}
\def\hu #1{\hbox{\hueca #1}}\def\hu #1{\hbox{\hueca #1}}
\begin{document}

\title{Weight modules over split Lie algebras}
\author{Antonio J. Calder\'{o}n Mart\'{\i }n
\thanks{Supported by the PCI of the UCA `Teor\'\i a de Lie y
Teor\'\i a de Espacios de Banach', by the PAI with project numbers
 FQM298, FQM2467, FQM3737 and by the project of the Spanish Ministerio de Educaci\'on y Ciencia
 MTM2007-60333.}\\
Jos\'{e} M.  S\'{a}nchez Delgado \\
Departamento de Matem\'{a}ticas. \\
Universidad de C\'{a}diz. 11510 Puerto Real, C\'{a}diz, Spain.\\
e-mail: ajesus.calderon@uca.es\\
e-mail: josemaria.sanchezdelgado@alum.uca.es }
\date{}
\maketitle

\begin{abstract}
We study the structure of  weight modules $V$ with  restrictions
neither on the dimension nor on the base field, over split Lie
algebras $L$.
We show that if $L$ is perfect and $V$ satisfies $LV=V$ and
${\mathcal Z}(V)=0$, then
$$\hbox{$L =\bigoplus\limits_{i\in I} I_{i}$ and $V = \bigoplus\limits_{j \in J} V_{j}$}$$
 with any $I_{i}$ an ideal of $L$
satisfying $[I_{i},I_{k}]=0$ if $i \neq k$, and any $V_{j}$ a
(weight) submodule of $V$ in such a way that for any $j \in J$
there exists a unique $i \in I$ such that $I_iV_j \neq 0,$ being
$V_j$  a weight module over  $I_i$. Under certain conditions, it
is shown that the above decomposition of $V$ is by means of the
family of its minimal submodules, each one being a simple (weight)
submodule.
\medskip

{\it Keywords}: Infinite dimensional Lie module, infinite
dimensional split Lie algebra, structure theory.

\end{abstract}

\section{Introduction and previous definitions}



Throughout this paper, weight modules $V$ and split Lie algebras
$L$ are considered of arbitrary dimensions and over an arbitrary
base field ${\hu K}$. It is worth to mention that, unless
otherwise stated, there is not any restriction on
$\dim{V}_{\gamma}$, $\dim{L}_{\alpha}$ or the products
$L_{\alpha}V_{\gamma}$  where $V_{\gamma}$ denotes the weight
space associated to the weight $\gamma$ of $V$ and $L_{\alpha}$
the root space associated to the root $\alpha$ of  $L$.

In $\S2$ we develop connection of weights techniques in the
framework of weight modules $V$ over split Lie algebras $L$ so as
to show that in case $LV=V$ and ${\mathcal Z}(V)=0$ then $V$
decomposes as the direct sum of an adequate family of nonzero
(weight) submodules of $V$, $V = \bigoplus\limits_{j\in J} V_j$.
In $\S3$ we slightly modify  the above mentioned connection
techniques of weights in order to apply them to the set of nonzero
roots of $L$, as consequence we show that if $L$ is perfect then
$L = \bigoplus\limits_{i\in I} I_i$ where any $I_i$ is a nonzero
ideal of $L$ satisfying $[I_i,I_k]=0$ if $i \neq k$. In $\S4$ we
relate the decompositions of $V$ and $L$ obtained in the previous
sections so as to get as main result that in case $LV=V$,
${\mathcal Z}(V)=0$ and $L$ is perfect then $L =
\bigoplus\limits_{i\in I} I_i$ and  $V = \bigoplus\limits_{j\in J}
V_j$ as above, in such a way that for any $j \in J$ there exists
one and only one $i\in I$ such that $I_iV_j \neq 0$,
 being any $V_j$ a weight module over $I_i$.  In the last section,
 $\S5$, we study  if any of the $V_j$, $j \in J$, in the above decomposition of $V$ is simple. Under certain
 conditions we give an affirmative answer. Finally, we would like
 to note the increasing interest in the study of weight modules
 over split Lie algebras, (and superalgebras),
 specially motivated by their relation with mathematical physics
 (see \cite{phy1, phy2, phy3, phy4, phy5, chino1, chino2}).

\medskip

Given an element $x$ of a Lie algebra $L,$ we denote  by  ${\rm
ad}_x$  the adjoint mapping ${\rm ad}_x$ defined as ${\rm
ad}_x(y):=[x,y]$ for any $y \in L$. A {\it splitting Cartan
subalgebra} $H$ of  ${L}$ is defined as a maximal abelian
subalgebra  of $L$, satisfying that the adjoint mappings ${\rm
ad}_h$, for $h \in H$, are simultaneously diagonalizable. If $L$
contains a splitting Cartan subalgebra $H$, then $L$ is called a
{\it split Lie algebra}, (see for instance \cite{Stu}). This means
that we have a root spaces decomposition
  $$L= H \oplus (\bigoplus\limits_{\alpha \in \Lambda}{L}_{\alpha})$$ where
$L_{\alpha}= \{v_{\alpha} \in L : [v_{\alpha},h] =
\alpha(h)v_{\alpha}$ for any $h \in H\}$ for a linear functional
$\alpha \in H^*$ and $\Lambda := \{\alpha \in H^* \backslash
\{0\}: L_{\alpha} \neq 0\}$. The subspaces $L_{\alpha}$ for
$\alpha \in H^*$ are called {\it root spaces} of $L$, (respect to
$H$),  and the elements $\alpha \in \Lambda \cup \{0\}$ are called
{\it roots} of $L$, (respect to $H$). Clearly  $L_0 = H$ and, as
consequence of Jacobi identity, $[L_{\alpha}, L_{\beta}] \subset
L_{\alpha + \beta}$ for any $\alpha, \beta \in \Lambda \cup
\{0\}$. We also say that $\Lambda$ is {\it symmetric} if for any
$\alpha \in \Lambda$ we have that  $-\alpha \in \Lambda$. Here we
note that there are many interesting examples of split Lie
algebras with a symmetric root system. For instance, we have the
separable semisimple  $L^*$-algebras (\cite{Schue1}), the
semisimple locally finite split Lie algebras over a  field of
characteristic zero (\cite{Neeb}), the generalized oscilator
algebras (\cite[Example 1.4, (a)]{Stu}) or the Virasoro algebras
(\cite[Example 1.4, (b)]{Stu}).

We will denote by ${\mathcal Z}(L)=\{e \in L: [e,L] = 0\}$ the
{\it center} of a Lie algebra $L$. We also recall that $L$ is
called {\it perfect} if ${\mathcal Z}(L)=0$ and $[L,L]=L$.



\begin{Definition}\label{def2}
Let $V$ be a  module over a  Lie algebra $L$ with splitting Cartan
subalgebra $H$.  For a linear functional $\gamma : H
\longrightarrow \hu{K},$  the {\rm weight space} of $V$, (respect
to $H$), associated with $\gamma$ is the subspace
$$V_{\gamma} = \{v_{\gamma} \in V : h  v_{\gamma} = \gamma(h)v_{\gamma}\hspace{0.2cm} {\it
for} \hspace{0.2cm} {\it any} \hspace{0.2cm} h \in H\}.$$ The
elements $\gamma \in H^*$ satisfying $V_{\gamma} \neq 0$ are
called {\rm weights} of $V$ respect to $H$ and we denote
$\mathcal{P} := \{\gamma \in H^* \backslash \{0\} : V_{\gamma}
\neq 0\}$. We say that $V$ is a {\rm weight module}, respect to
$H$, if
$$V = V_0 \oplus (\bigoplus\limits_{\gamma \in \mathcal{P}}{V}_{\gamma}).$$
We also say that $\mathcal{P}$ is the {\rm weight system} of $V$.
\end{Definition}
The weight system $\mathcal{P}$ is called {\it symmetric} if for
any $\gamma \in \mathcal{P}$ we have that $-\gamma \in
\mathcal{P}$.

Split Lie algebras are examples of weight modules over themselves,
where
 $ \mathcal{P}=\Lambda$ and  $V_{\gamma} = L_{\gamma}$ for $\gamma \in \mathcal{P} \cup
 \{0\}$. Since the even part $L^0$ of the standard embedding of a
 split Lie triple system $T$ and of a split twisted inner derivation
 triple system $M$ is a split Lie algebra, the natural actions of
 $L^0$ over $T$ and $M$ make of $T$ and $M$ weight modules over
 the split Lie algebra $L^0$. So the present paper extend the
 results in \cite{Yoalg, Yotriple1, Twisted}.  We also have
 that any split Lie superalgebra, see \cite{Super}, $L=L^{\bar
 0} \oplus L^{\bar 1}$ is a weight module over the split Lie
 algebra $L^{\bar
 0}$. These examples provide us of interesting applications of the results in the paper which can be interpreted
 from a common viewpoint, the one of weight modules over a split Lie algebra. We devote  Section 6 to develop these examples in
 detail. We also remark in $\S 6$ future perspectives of the
 matter.

\section{Connections of weights. Decompositions of $V$}

From now on, (throughout  the paper), $$V = V_0 \oplus
(\bigoplus\limits_{\gamma \in \mathcal{P}}{V}_{\gamma})$$ denotes
a weight module with a symmetric weight system $\mathcal{P}$,
respect to a split Lie algebra $$L = H \oplus
(\bigoplus\limits_{\alpha \in \Lambda}L_{\alpha})$$ with a
symmetric root  system $\Lambda$.

\medskip

 In order to clarify the results along the paper let us
introduce a concrete example which is of potential interest  in
the modelling  of physical problems involving split Lie algebras
and some involutive characteristic. After any main result, we will
refer to this example  to illustrate it.

\begin{example}\label{example}
 Let ${\frak L}=H\oplus \bigoplus\limits_{\alpha \in
\Lambda} {\frak L}_{\alpha}$ be a split Lie algebra over a base
field of characteristic distinct to two with a symmetric root
system, and $\xi $ an involutive automorphism of ${\frak L}$.
Hence we can write
$${\frak L}=Sym({\frak L},\xi) \oplus Skw({\frak L},\xi),$$
where $Sym({\frak L},\xi)=\{e\in {\frak L}: \xi(e)=e\}$ and
$Skw({\frak L},\xi)=\{e\in {\frak L}: \xi(e)=-e\}.$
 Denote by $\Pi_0:{\frak L} \to
Sym({\frak L},\xi)$ and $\Pi_1:{\frak L} \to Skw({\frak L},\xi)$
the projection maps $\Pi_i(x_0+x_1)=x_i$. Suppose   $\xi(H)
\subset H$ and  $\Pi_i({\frak L}_{\alpha})\neq 0$ for any $\alpha
\in \Lambda \cup \{0\}$, $i \in \{0,1\}$. Then it is
straightforward to verify that $$L:=Sym({\frak L},\xi)$$ is a
split Lie algebra respect to the splitting Cartan subalgebra
$Sym(H,\xi)$ with set of nonzero roots
$\Lambda|_{Sym(H,\xi)}=\{{\alpha}|_{Sym(H,\xi)}: \alpha \in
\Lambda\}$ and with nonzero root spaces
$$L_{{\alpha}|_{Sym(H,\xi)}}=\Pi_0({\frak L}_{\alpha})$$
 and that
$$V:=Skw({\frak L},\xi )$$ is a
weight module respect to the split Lie algebra $L$ under the
natural action
$$Sym({\frak L},\xi) \times Skw({\frak L},\xi) \to Skw({\frak
L},\xi)$$
$$(x,y)\mapsto [x,y],$$
 with set of nonzero weights
$\mathcal{P}=\Lambda=\{{\alpha}|_{Sym(H,\xi)}: \alpha \in
\Lambda\}$, with nonzero weight spaces
$$V_{{\alpha}|_{Sym(H,\xi)}}=\Pi_1({\frak L}_{\alpha})$$
and with $V_0=Skw(H,\xi)$.
 \end{example}

 Let us return to our study of weight modules over split Lie algebras  by developing
connections of weights techniques in this framework.

\begin{Definition}\label{con}
Let $\gamma$ and $\delta$ be two nonzero weights. We say that
$\gamma$ is {\rm connected} to $\delta$ if there exist
$\alpha_1,...,\alpha_n \in \Lambda$ such that

\begin{enumerate}
\item[{\rm 1.}] $\{ \gamma+\alpha_1, \gamma+\alpha_1+\alpha_2,...,
\gamma+\alpha_1+\alpha_2+\cdots+\alpha_{n-1}\} \subset
\mathcal{P}$,

\item[{\rm 2.}] $\gamma+\alpha_1+\alpha_2+\dots+\alpha_n \in
\{\delta, -\delta\}$,
\end{enumerate}
where the sums are considered in $H^*$.

 We also say that
$\{\gamma,\alpha_1,..., \alpha_n\}$ is a {\rm connection} from
$\gamma$ to $\delta$.
\end{Definition}

\begin{note}\label{note}
For an easier notation, we will understand that $\{\gamma\}$ is a
connection from $\gamma$ to itself and to $-\gamma$.

\end{note}

The next result shows the connection relation is of equivalence.

\begin{proposition}\label{pro1}
 The relation $\sim$ in $\mathcal{P}$ defined
by $\gamma \sim \delta$ if and only if $\gamma$ is connected to
$\delta$ is an equivalence relation.
\end{proposition}

\begin{proof}
By Note \ref{note}, $\{\gamma\}$ is a connection from $\gamma$ to
itself and therefore $\gamma \sim \gamma$.

Let us see the symmetric character of $\sim$: If $\gamma \sim
\delta$, there exists a connection
$\{\gamma,\alpha_1,...,\alpha_n\} $ from $\gamma$ to $\delta$,
being so
$$\{\gamma + \alpha_1, \gamma + \alpha_1 + \alpha_2, ... , \gamma
+ \alpha_1 + \cdots + \alpha_{n-1}\} \subset \mathcal{P}$$ and
$\gamma + \alpha_1 +\cdots+\alpha_n \in \{\delta, -\delta\}.$
Hence, we can distinguish two possibilities. In the first one
$\gamma + \alpha_1 + \cdots+ \alpha_n = \delta,$ and in the second
one $\gamma + \alpha_1 + \cdots + \alpha_n = -\delta.$
 Now  observe that   the
set $\{\delta,-\alpha_n,
 -\alpha_{n-1},  ... , -\alpha_1\}$  gives us a connection from  $\delta$ to $\gamma$ if we
have the first possibility and   $\{\delta,\alpha_n,
 \alpha_{n-1},  ... , \alpha_1\}$ if we have the second one. Hence  $\sim$ is symmetric.

Finally, suppose $\gamma \sim \delta$ and $\delta \sim \eta$, and
 write $\{\gamma,\alpha_1,...,\alpha_n\}$ for a connection from
$\gamma$ to $\delta$ and $\{\delta,\beta_1,...,\beta_m\}$ for a
connection from $\delta$ to $\eta$. If $\delta \neq \pm \eta$,
then $m \geq 1$ and  so $\{\gamma,
\alpha_1,...,\alpha_n,\beta_1,..., \beta_m\}$ is a connection from
$\gamma$ to $\eta$ in case $\gamma+\alpha_1+\cdots+
\alpha_n=\delta$, and $\{\gamma,\alpha_1,...,
\alpha_n,-\beta_1,...,-\beta_m\}$ in case $\gamma+\alpha_1+\cdots+
\alpha_n=-\delta$. If  $\delta \in \{\eta, -\eta\}$ then
$\{\gamma,\alpha_1,..., \alpha_n\}$ is a connection from $\gamma$
to $\eta$. Therefore $\gamma \sim \eta$ and $\sim$ is of
equivalence.
\end{proof}

Given $\gamma \in \mathcal{P}$, we denote by
$$\mathcal{P}_{\gamma} := \{\delta \in \mathcal{P} : \delta \sim \gamma\}.$$
Clearly if $\delta \in \mathcal{P}_{\gamma}$ then $-\delta \in
\mathcal{P}_{\gamma}$ and, by Proposition \ref{pro1}, if $\eta
\notin \mathcal{P}_{\gamma}$ then $\mathcal{P}_{\gamma} \cap
\mathcal{P}_{\eta} = \emptyset$.

Our next goal is to associate an (adequate) weight submodule
$V_{\mathcal{P}_{\gamma}}$ to any $\mathcal{P}_{\gamma}$. For
$\mathcal{P}_{\gamma}, \gamma \in \mathcal{P},$ we define the
following linear subspace of $V$:

$$V_{\mathcal{P}_{\gamma}} := (\sum\limits_{\alpha \in \Lambda \cap
\mathcal{P}_{\gamma}}L_{-\alpha}V_{\alpha}) \oplus
(\bigoplus\limits_{\delta \in \mathcal{P}_{\gamma}}V_{\delta}).$$

(We also denote by
$V_{0,\mathcal{P}_{\gamma}}:=\sum\limits_{\alpha \in \Lambda \cap
\mathcal{P}_{\gamma}}L_{-\alpha}V_{\alpha} \subset V_0$).

\begin{lemma}\label{pag5}

The following assertions hold:
\begin{enumerate}
\item  For any $\alpha \in \Lambda$ and  $\gamma \in \mathcal{P}$
with  $\alpha \neq - \gamma$, if   $L_{\alpha}V_{\gamma} \neq 0$
then $\alpha+ \gamma \sim \gamma$.

\item For any $\alpha, \beta \in \Lambda \cap \mathcal{P}$, if
$L_{\beta}(L_{-\alpha}V_{\alpha}) \neq 0$  then $\alpha \sim
\beta$.
\end{enumerate}
\end{lemma}
\begin{proof}
1. The fact $L_{\alpha}V_{\gamma} \neq 0$ with  $\alpha \neq -
\gamma$ ensures $\alpha + \gamma \in \mathcal{P}$. Hence, just
consider the connection $\{\gamma, \alpha\}$.

2. If $\beta = \pm \alpha$ it is clear. Hence suppose $\beta \neq
\pm \alpha$. From $L_{\beta}(L_{-\alpha}V_{\alpha}) \neq 0$ we
have either $[L_{\beta},L_{-\alpha}]V_{\alpha} \neq 0$  or
$L_{-\alpha}(L_{\beta}V_{\alpha}) \neq 0$. In the first case
$\beta-\alpha \in \Lambda$ and $\{\alpha, \beta-\alpha\}$ is a
connection form $\alpha$ to $\beta$. In the second case we have
$\alpha + \beta \in \mathcal{P}$ and then $\{\alpha, \beta,
-\alpha\}$ is a connection form $\alpha$ to $\beta$.
\end{proof}

We recall that a Lie module $V$ is said to be {\it simple} if its
only submodules are $\{0\}$ and $V$.

\begin{theorem}\label{teo1}
Let $\gamma \in \mathcal{P}$. Then the following assertions hold.

\begin{enumerate}
\item[{\rm 1.}] $V_{\mathcal{P}_{\gamma}}$ is a weight submodule
of $V$.

\item[{\rm 2.}] If $V$ is simple, then there exists a connection
from $\gamma$ to $\delta$ for any $\gamma, \delta \in \mathcal{P}$
and $V_0 = \sum\limits_{\alpha \in \Lambda \cap
\mathcal{P}_{\gamma}}L_{-\alpha}V_{\alpha}$.
\end{enumerate}
\end{theorem}

\begin{proof}
1. For any $e\in L$ and for any $v \in V_{\mathcal{P}_{\gamma}}$
we can write  $e = h + \sum\limits_{i=1}^ne_{\alpha_i}$, with $h
\in H$,
 $e_{\alpha_i} \in L_{\alpha_i}$ and $\alpha_i \in \Lambda$; and  $v = v_0
+ \sum\limits_{j=1}^mv_{\gamma_j}$ with $v_0 \in
\sum\limits_{\alpha \in \Lambda \cap
\mathcal{P}_{\gamma}}L_{-\alpha}V_{\alpha}$, $v_{\gamma_j} \in
V_{\gamma_j},$ and $ \gamma_j \in \mathcal{P}_{\gamma}.$

From here

\begin{equation}\label{submodulo_proof}
ev   \in (\bigoplus\limits_{j=1}^mV_{\gamma_j}) +
(\sum\limits_{i=1}^nL_{\alpha_i}(L_{-\alpha}V_{\alpha})) +
(\bigoplus\limits_{i=1,j=1}^{n,m}L_{\alpha_i}V_{\gamma_j}).
\end{equation}
 If $L_{\alpha_i}(L_{-\alpha}V_{\alpha})
\neq 0$, then  $\alpha_i \in \mathcal{P}$ and so Lemma
\ref{pag5}-2 gives us  $\alpha \sim \alpha_i$. We have   $\alpha_i
\in \mathcal{P}_{\gamma}$ and we get
$$\sum\limits_{i=1}^nL_{\alpha_i}(L_{-\alpha}V_{\alpha}) \subset
\sum\limits_{i=1}^nV_{\alpha_i} \subset
V_{\mathcal{P}_{\gamma}}.$$ Consider now the third summand in
(\ref{submodulo_proof}). If $\alpha_i=-\gamma_j$ then

\begin{equation}\label{eq2}
L_{\alpha_i}V_{\gamma_j}=L_{-\gamma_j}V_{\gamma_j} \subset
V_{\mathcal{P}_{\gamma}}.
\end{equation}
 If $\alpha_i \neq-\gamma_j$, Lemma
\ref{pag5}-1 gives us   $\alpha_i + \gamma_j \sim \gamma_j $.
Hence $\alpha_i + \gamma_j \in \mathcal{P}_{\gamma}$ and so
$$L_{\alpha_i}V_{\gamma_j}\subset V_{\alpha_i+\gamma_j}
\subset V_{\mathcal{P}_{\gamma}}.$$ From here,  and taking into
account equation (\ref{eq2}), we get
$\bigoplus\limits_{i=1,j=1}^{n,m}L_{\alpha_i}V_{\gamma_j} \subset
V_{\mathcal{P}_{\gamma}}$ which  completes the proof of 1.

2. The simplicity of $V$ implies $V_{\mathcal{P}_{\gamma}} = V$
for any $\gamma \in \mathcal{P}$. Therefore $\mathcal{P}_{\gamma}
= \mathcal{P}$,  and so $V$ has all its nonzero weights connected,
and $V_0 = \sum\limits_{\alpha \in \Lambda \cap
\mathcal{P}_{\gamma}}L_{-\alpha}V_{\alpha}$.
\end{proof}

Theorem \ref{teo1}-1 let us assert that for any $\gamma \in
\mathcal{P}$, $V_{\mathcal{P}_{\gamma}}$ is a weight submodule of
$V$ that we call the submodule of $V$ {\it associated} to
$\mathcal{P}_{\gamma}$.

\begin{proposition}\label{teo2}
For a linear  complement $\mathcal{U}$ of $span_{\hu
K}\{L_{-\alpha}V_{\alpha} : \alpha \in \Lambda \cap \mathcal{P}\}$
in $V_0$, we have the sum of ${\hu K}$-vector subspaces

$$V = \mathcal{U} + (\sum\limits_{[\gamma] \in \mathcal{P}/\sim} V_{[\gamma]}),$$

where any $V_{[\gamma]}$ is one of the weight submodules described
in Theorem \ref{teo1}-1.
\end{proposition}

\begin{proof}
By Proposition \ref{pro1}, we can consider the quotient set
$$\mathcal{P}/\sim:= \{[\gamma] : \gamma \in \mathcal{P}\}.$$ Let us
denote by $V_{[\gamma]} := V_{\mathcal{P}_{\gamma}}$. We have
$V_{[\gamma]}$ is well defined and by Theorem \ref{teo1}-1 a
weight  submodule of $V$. Therefore $V = \mathcal{U} +
(\sum\limits_{[\gamma] \in \mathcal{P}/\sim} V_{[\gamma]}).$
\end{proof}

Recall that the {\it center} of $V$ is defined as  the set
${\mathcal Z}(V)=\{v \in V: L  v = 0\}$.

\begin{theorem}\label{corolario}
If $LV = V$ and ${\mathcal Z}(V)=0$, then $V$ is the direct sum of
the weight submodules  given in Proposition \ref{teo2}, $$V =
\bigoplus\limits_{[\gamma] \in \mathcal{P}/\sim} V_{[\gamma]}.$$
\end{theorem}

\begin{proof}
Taking into account Proposition \ref{teo2}, it is clear that the
fact  $LV = V$  gives us $V = \sum\limits_{[\gamma] \in
\mathcal{P}/\sim} V_{[\gamma]}.$ To verify the direct character of
this sum,  we just have to prove that the sum $V_0=
\sum\limits_{[\gamma] \in \mathcal{P}/\sim}V_{0,[{\gamma}]}$,
where $V_{0,[{\gamma}]}=\sum\limits_{\beta \in \Lambda \cap
[{\gamma}]} L_{-\beta}V_{\beta}$ is direct. So, take $v_0 \in
V_{0,[{\gamma}]} \cap (\sum\limits_{[\gamma_i] \in
\mathcal{P}/\sim \setminus [\gamma]}V_{0,[{\gamma_i}]})$. Since
${\mathcal Z}(V)=0$, $L$ is split and $HV_0=0$, there exists
$0\neq e_{\alpha} \in L_{\alpha}$, $\alpha \in \Lambda$, such that
$e_{\alpha} v_0 \neq 0$, being then $\alpha \in \Lambda \cap
\mathcal{P}$. By Lemma \ref{pag5}-2,  we have $\alpha \sim \gamma$
and $\alpha \sim \gamma_i$,  for some $i$,  with $[\gamma] \neq
[\gamma_i]$ which is a contradiction. Therefore we can write $V =
\bigoplus\limits_{[\gamma] \in \mathcal{P}/\sim} V_{[\gamma]}.$
\end{proof}

\medskip

Let us concrete this result in Example \ref{example}. We have that
in case $$[Sym({\frak L},\xi),Skw({\frak L},\xi)]=Skw({\frak
L},\xi)$$ and $\{v \in Skw({\frak L},\xi):[Sym({\frak
L},\xi),v]\}=0$, we can decompose  $Skw({\frak L},\xi)$ by means
of an adequate family of weight submodules as
$$Skw({\frak L},\xi)=
\bigoplus\limits_{[\alpha] \in \Lambda|_{Sym(H,\xi)/\sim}}
Skw({\frak L},\xi)_{[\alpha]}.$$

\section{Connections of roots. Decompositions of $L$}

We begin this section by introducing a concept of connections of
roots for $L$ in a slightly different way to the one of connection
of weights for $\mathcal{P}$ developed in the previous section. To
do that, we will connect the nonzero roots of $L$ through nonzero
roots of $L$ and nonzero weights of $\mathcal{P}$ considered both
as elements in $H^*$.


\begin{Definition}\label{def nueva conec}
Let $\alpha, \beta$ be two nonzero roots of $L$. We say that
$\alpha$ is {\rm connected} to $\beta$ if there exist
$\zeta_1,...,\zeta_n \in {\Lambda}\cup \mathcal{P}$ such that

\begin{enumerate}
\item[{\rm 1.}] $\alpha = \zeta_1$,

\item[{\rm 1.}] $\{\zeta_1, \zeta_1 + \zeta_2, ... , \zeta_1 +
\cdots + \zeta_{n-1}\} \subset \Lambda \cup \mathcal{P}$,

\item[{\rm 2.}] $\zeta_1+\cdots+\zeta_n \in \{\beta, -\beta\}$,
\end{enumerate}
where the sums are considered in $H^*$.

 We also say that
$\{\zeta_1,...,\zeta_n\}$ is a {\rm connection} from $\alpha$ to
$\beta$.
\end{Definition}

We can prove that the connection relation is of equivalence in
$\Lambda$ by arguing as in the proof of Proposition \ref{pro1}. So
we can assert.

\begin{proposition}\label{rel_eq_nueva_conec}
The relation $\approx$ in $\Lambda$ defined by $\alpha \approx
\beta$ if and only if $\alpha$ is connected to $\beta$ is an
equivalence relation.
\end{proposition}

Given $\alpha \in \Lambda$, we denote by
$$\Lambda_{\alpha}:=\{\beta \in \Lambda : \beta
\approx \alpha\}.$$ We also have that if $\beta \in
\Lambda_{\alpha}$ then $-\beta \in \Lambda_{\alpha}$ and, by
Proposition \ref{rel_eq_nueva_conec}, if $\mu \notin
\Lambda_{\alpha}$ then $\Lambda_{\alpha} \cap \Lambda_{\mu} =
\emptyset$.

Our next aim is to associate an adequate ideal of $L$ to any
$\Lambda_{\alpha}$. For $\Lambda_{\alpha}, \alpha \in \Lambda$, we
define
$$H_{\Lambda_{\alpha}}: = span_{\hu
K}\{[L_{\beta},L_{-\beta}]: \beta \in \Lambda_{\alpha}\} \subset
H,$$ and
$$N_{\Lambda_{\alpha}}: = \bigoplus\limits_{\beta
\in \Lambda_{\alpha}}L_{\beta}.$$ We denote by
$L_{\Lambda_{\alpha}}$ the following linear subspace of $L$,
$$L_{\Lambda_{\alpha}} :=
H_{\Lambda_{\alpha}} \oplus N_{\Lambda_{\alpha}}.$$

\begin{proposition}\label{pro1_nueva}
Let $\alpha \in \Lambda$. Then the following assertions hold.

\begin{enumerate}
\item[{\rm 1.}] $[L_{\Lambda_{\alpha}},L_{\Lambda_{\alpha}}]
\subset L_{\Lambda_{\alpha}}$.

\item[{\rm 2.}]  If $\mu \notin \Lambda_{\alpha}$ then
$[L_{\Lambda_{\alpha}}, L_{\Lambda_{\mu}}] = 0$.
\end{enumerate}
\end{proposition}

\begin{proof}
1. Taking into account $H=L_0$ and $[L_{\alpha},L_{\beta}] \subset
L_{\alpha+\beta}$, we have
\begin{equation}\label{cero}
[L_{\Lambda_{\alpha}},L_{\Lambda_{\alpha}}]=[H_{\Lambda_{\alpha}}
\oplus N_{\Lambda_{\alpha}}, H_{\Lambda_{\alpha}} \oplus
N_{\Lambda_{\alpha}}] \subset N_{\Lambda_{\alpha}} +
\sum\limits_{\beta, \lambda \in
\Lambda_{\alpha}}[L_{\beta},L_{\lambda}].
\end{equation}
If $\lambda = -\beta$ then
\begin{equation}\label{uno}
[L_{\beta},L_{\lambda}] \subset H_{\Lambda_{\alpha}}.
\end{equation}
If $\lambda \neq -\beta$ and $[L_{\beta},L_{\lambda}] \neq 0$,
then $\beta + \lambda \in \Lambda$. From here, if
$\{\zeta_1,...,\zeta_n\}$ is a connection from $\alpha$ to $\beta$
then $\{\zeta_1,...,\zeta_n,\lambda\}$ is a connection from
$\alpha$ to $\beta + \lambda$ in case $\zeta_1+\cdots+\zeta_n =
\beta$ and $\{\zeta_1,.....,\zeta_n,-\lambda\}$ in case
$\zeta_1+\cdots+\zeta_n= -\beta$. Hence $\beta + \lambda \in
\Lambda_{\alpha}$ and so
\begin{equation}\label{dos}
[L_{\beta},L_{\lambda}]\subset N_{\Lambda_{\alpha}}.
\end{equation}
From equations (\ref{cero}), (\ref{uno}) and (\ref{dos}) we
conclude $[L_{\Lambda_{\alpha}},L_{\Lambda_{\alpha}}] \subset
L_{\Lambda_{\alpha}}$.

2.  We have
$$[L_{\Lambda_{\alpha}}, L_{\Lambda_{\mu}}] = [H_{\Lambda_{\alpha}}
\oplus N_{\Lambda_{\alpha}}, H_{\Lambda_{\mu}} \oplus
N_{\Lambda_{\mu}}] \subset$$
\begin{equation}\label{cuatro}
\subset [H_{\Lambda_{\alpha}}, N_{\Lambda_{\mu}}] +
[N_{\Lambda_{\alpha}}, H_{\Lambda_{\mu}}] + [N_{\Lambda_{\alpha}},
N_{\Lambda_{\mu}}].
\end{equation}
Consider in (\ref{cuatro}) the third summand
$[N_{\Lambda_{\alpha}}, N_{\Lambda_{\mu}}]$ and suppose there
exist $\beta \in \Lambda_{\alpha}$ and $\eta \in \Lambda_{\mu}$
such that $[L_{\beta}, L_{\eta}] \neq 0$. As necessarily $\beta
\neq -\eta$, then $\beta + \eta \in \Lambda$. So $\{\beta, \eta,
-\beta\}$ is a connection between $\beta$ and $\eta$. By the
transitivity of the connection relation we have $\mu \in
\Lambda_{\alpha}$, a
 contradiction. Hence $[L_{\beta},L_{\eta}] = 0$ and so
\begin{equation}\label{nueve}
[N_{\Lambda_{\alpha}},N_{\Lambda_{\mu}}]=0.
\end{equation}
Consider now the first summand $[H_{\Lambda_{\alpha}},
N_{\Lambda_{\mu}}]$ in equation (\ref{cuatro}) and suppose there
exist $\beta \in \Lambda_{\alpha}$ and $\eta \in \Lambda_{\mu}$
such that $[[L_{\beta},L_{-\beta}],L_{\eta}] \neq 0$. By applying
 Jacobi identity, either   $[[L_{-\beta},L_{\eta}],L_{\beta}]
\neq 0$ or $[[L_{\eta},L_{\beta}], L_{-\beta}] \neq 0$ and so
either $[L_{-\beta},L_{\eta}] \neq 0$ or $[L_{\eta},L_{\beta}]
\neq 0$, what contradicts equation (\ref{nueve}).  Hence $$[
H_{\Lambda_{\alpha}},N_{\Lambda_{\mu}}] = 0.$$ Finally, we note
that the same above argument shows
$$[N_{\Lambda_{\alpha}}, H_{\Lambda_{\mu}}]=0.$$
By equation (\ref{cuatro}) we conclude
$[L_{\Lambda_{\alpha}},L_{\Lambda_{\mu}}]=0$.
\end{proof}

By Proposition \ref{pro1_nueva}-1 we can assert that for any
$\alpha \in \Lambda$, $L_{\Lambda_{\alpha}}$ is a subalgebra of
$L$ that we call the subalgebra of $L$ {\it associated} to
$\Lambda_{\alpha}$.

\begin{theorem}\label{teo1_nueva}
The following assertions hold.

\begin{enumerate}
\item[{\rm 1.}]  For any $\alpha \in \Lambda$, the subalgebra
\[
L_{\Lambda_{\alpha}} = H_{\Lambda_{\alpha}} \oplus
N_{\Lambda_{\alpha}}
\]
of $L$ associated to $\Lambda_{\alpha}$ is an ideal of $L$.

\item[{\rm 2.}]  If $L$ is simple, then there exists a connection
from $\alpha$ to $\beta$ for any $\alpha,\beta \in \Lambda$ and $H
= \sum\limits_{\alpha \in \Lambda}[L_{\alpha}, L_{-\alpha}]$.
\end{enumerate}
\end{theorem}

\begin{proof}
1.  Since $[L_{\Lambda_{\alpha}},H] \subset N_{\Lambda_{\alpha}}$,
taking into account Proposition \ref{pro1_nueva}, we have
$$
\lbrack L_{\Lambda_{\alpha}},L]=[L_{\Lambda_{\alpha}}, H
\oplus(\bigoplus \limits_{\beta \in
\Lambda_{\alpha}}L_{\beta})\oplus (\bigoplus \limits_{\mu \notin
\Lambda_{\alpha}}L_{\mu})]\subset L_{\Lambda_{\alpha}}.
$$

2. The simplicity of $L$ implies $L_{\Lambda_{\alpha}}=L$.
Therefore $\Lambda_{\alpha} = \Lambda$ and $H =
\sum\limits_{\alpha \in \Lambda}[L_{\alpha}, L_{-\alpha}]$.
\end{proof}

\begin{proposition}\label{teo2_nueva}  For a linear  complement $\mathcal{U}$ of $span_{\hu
K}\{[L_{\alpha},L_{-\alpha}]:\alpha\in \Lambda\}$ in $H$, we have
$$L = \mathcal{U}+ \sum\limits_{[\alpha] \in
\Lambda/\approx} I_{[\alpha]},$$ where any $I_{[\alpha]}$ is one
of the ideals $L_{\Lambda_{\alpha}}$ of $L$ described in Theorem
\ref{teo1_nueva}-1, satisfying $[I_{[\alpha]},I_{[\beta]}]=0$ if
$[\alpha] \neq [\beta].$
\end{proposition}

\begin{proof}
The decomposition of $L$ can be obtained as the one of $V$ in
Proposition \ref{teo2}. Finally, by  applying Proposition
\ref{pro1_nueva}-2 we also obtain $[I_{[\alpha]},I_{[\beta]}]=0$
if $[\alpha] \neq [\beta].$
\end{proof}

\begin{theorem}\label{coro1}
If ${\mathcal Z}(L)=0$ and $H = \sum\limits_{\alpha \in
\Lambda}[L_{\alpha}, L_{-\alpha}]$, then $L$ is the direct sum of
the ideals given in Theorem \ref{teo1_nueva},
$$L =\bigoplus\limits_{[\alpha] \in \Lambda/\approx}
I_{[\alpha]}.$$
\end{theorem}

\begin{proof}
By Proposition \ref{teo2_nueva}, $L = \mathcal{U}+
\sum\limits_{[\alpha] \in \Lambda/\approx} I_{[\alpha]}$. From $H
= \sum\limits_{\alpha \in \Lambda}[L_{\alpha}, L_{-\alpha}]$ it is
clear that $\mathcal{U} = 0$. Finally, the direct character of the
sum now follows from the facts $[I_{[\alpha]},I_{[\beta]}]=0$, if
$[\alpha] \neq [\beta]$, and ${\mathcal Z}(L)=0$.
\end{proof}

\medskip

 Let us also concrete this result in Example \ref{example}.
We have that in case $\{e \in Sym({\frak L},\xi):[e, Sym({\frak
L},\xi)]\}=0$ and $$Sym(H,\xi)= \sum\limits_{{\alpha} \in
\Lambda}[\Pi_0({\frak L}_{\alpha}), \Pi_0({\frak L}_{-\alpha})],$$
we can decompose  $Sym({\frak L},\xi)$ by means of an adequate
family of  ideals as
$$Sym({\frak L},\xi) =\bigoplus\limits_{[\alpha] \in
\Lambda|_{Sym(H,\xi)/\approx}}
Sym({\frak L},\xi)_{[\alpha]}.$$

\section{Relating the decompositions of $V$ and $L$}

Suppose that  $V$  satisfies $LV=V$ and ${\mathcal Z}(V)=0$, and
that $L$  is perfect. Since the fact $[L,L]=L$ implies $H =
\sum\limits_{\alpha \in \Lambda}[L_{\alpha}, L_{-\alpha}]$,
Theorems \ref{corolario} and \ref{coro1} give us
$$\hbox{$V = \bigoplus\limits_{[\gamma] \in \mathcal{P}/\sim}
V_{[\gamma]}$ and $L =\bigoplus\limits_{[\alpha] \in
\Lambda/\approx} I_{[\alpha]}$}$$ with any $V_{[\gamma]}$ a weight
submodule of $V$ and any $I_{[\alpha]}$ an ideal of $L$ satisfying
$[I_{[\alpha]},I_{[\beta]}]=0$ if $[\alpha] \neq [\beta]$.

For any $[\gamma] \in \mathcal{P}/\sim$, since ${\mathcal
Z}(V)=0$, there exists $[\alpha] \in \Lambda/\approx$ such that
$I_{[\alpha]}V_{[\gamma]} \neq 0$.
 Let us show that this
$[\alpha]$ is unique. To do that, consider $\beta \in \Lambda$
satisfying  $I_{[\beta]}V_{[\gamma]} \neq 0$. From the facts
$I_{[\alpha]}V_{[\gamma]} \neq 0$ and $I_{[\beta]}V_{[\gamma]}
\neq 0$ we can take $\alpha' \in {[\alpha]}$, $\beta' \in
{[\beta]}$ and $\gamma',\gamma'' \in {[\gamma]}$ such that
$L_{\alpha'}V_{\gamma'} \neq 0$ and $L_{\beta'}V_{\gamma''} \neq
0$. Since $\gamma',\gamma'' \in {[\gamma]}$, we can fix a
connection $$\{\gamma', \alpha_1,...,\alpha_n\},$$
($\alpha_1,...,\alpha_n \in \Lambda$),  from the weight $\gamma'$
to the weight $\gamma''$.

Let us distinguish four cases. First, if $\alpha'+\gamma' \neq 0$,
$\beta'+\gamma'' \neq 0$ and so $\alpha'+\gamma'$,
$\beta'+\gamma'' \in \mathcal{P}$, we have that
$$\{\alpha',\gamma', -\alpha',\alpha_1,...,\alpha_n, \beta',
-\gamma''\}\subset \Lambda \cup \mathcal{P}$$ is a connection from
the root $\alpha'$ to the root $\beta'$ in the case $\gamma'+
\alpha_1+ \cdots +\alpha_n= \gamma''$ while $\{\alpha',\gamma',
-\alpha',\alpha_1,...,\alpha_n, -\beta', \gamma''\}$ gives us the
same connection of roots in the case $\gamma'+ \alpha_1+ \cdots
+\alpha_n= -\gamma''$. From here $\alpha' \approx \beta'$ and so
$[\alpha]=[\beta]$. Second, if $\alpha'+\gamma' = 0$,
$\beta'+\gamma'' \neq 0$ and so $\alpha'=-\gamma'$,
$\beta'+\gamma'' \in \mathcal{P}$, we have that
$$\{-\gamma',-\alpha_1,...,-\alpha_n,-\beta', \gamma''\}$$ is a
connection of roots between $\alpha'$ and $\beta'$ in the case
$\gamma'+ \alpha_1+ \cdots +\alpha_n= \gamma''$ while
$\{-\gamma',-\alpha_1,...,-\alpha_n,\beta', -\gamma''\}$ it is in
the case $\gamma'+ \alpha_1+ \cdots +\alpha_n= -\gamma''$. From
here,  $[\alpha]=[\beta]$. Third, if $\alpha'+\gamma'\neq  0$,
$\beta'+\gamma'' = 0$ we can argue as in the previous case to get
$[\alpha]=[\beta]$. Finally, in the fourth case we suppose
$\alpha'+\gamma'=  0$, $\beta'+\gamma'' = 0$ and so
$\alpha'=-\gamma'$, $\beta'=-\gamma''.$ Then
$$\{-\gamma',-\alpha_1,...,-\alpha_n\}$$ is a connection of roots
between $\alpha'$ and $\beta'$ which  implies $[\alpha]=[\beta]$.
We conclude that the element $[\alpha] \in \Lambda/\approx$ such
that $I_{[\alpha]}V_{[\gamma]} \neq 0$ is unique. Also, let us
observe that Lemma \ref{pag5} gives us that $V_{[\gamma]}$ is a
weight module over $I_{[\alpha]}$. Hence we can assert.

\begin{theorem}\label{lema_final}
Let $V$ be weight  module respect to  a perfect split Lie algebra
$L$ such that $LV=V$ and ${\mathcal Z}(V)=0$. Then
$$\hbox{$L =\bigoplus\limits_{i\in I} I_{i}$ and $V = \bigoplus\limits_{j \in J} V_{j}$}$$
 with any $I_{i}$ a nonzero ideal of $L$
satisfying $[I_{i},I_{k}]=0$ if $i \neq k$, and any $V_{j}$ a
nonzero weight submodule of $V$ in such a way that for any $j \in
J$ there exists a unique $i \in I$ such that $$I_iV_j \neq 0.$$
Furthermore $V_j$ is a weight module over  $I_i$.
\end{theorem}

\medskip

 Finally, let us  concrete this result in Example \ref{example} as
 always. Observe that in this case ${\mathcal P}= \Lambda$ and so
 the equivalence relations $\sim$ and $\approx$ in Definitions
 \ref{con} and \ref{def nueva conec} agree. From  here, we have
 that  in case $Sym({\frak L},\xi)$ is perfect, $[Sym({\frak L},\xi),Skw({\frak L},\xi)]=Skw({\frak
L},\xi)$ and $\{v \in Skw({\frak L},\xi):[Sym({\frak
L},\xi),v]\}=0$, we can  decompose  $Sym({\frak L},\xi)$ and
$Skw({\frak L},\xi)$ by means of  adequate  families of ideals and
weight submodules respectively

$$Sym({\frak L},\xi) =\bigoplus\limits_{[\alpha] \in \Lambda|_{Sym(H,\xi)/\sim}}
Sym({\frak L},\xi)_{[\alpha]}$$ and $$Skw({\frak L},\xi)=
\bigoplus\limits_{[\alpha] \in \Lambda|_{Sym(H,\xi)/\sim}}
Skw({\frak L},\xi)_{[\alpha]}$$ in such a way that for any
$Skw({\frak L},\xi)_{[\alpha]}$, precisely $Sym({\frak
L},\xi)_{[\alpha]}$ is the unique  ideal in the decomposition of
$Sym({\frak L},\xi)$ such that $[Sym({\frak
L},\xi)_{[\alpha]},Skw({\frak L},\xi)_{[\alpha]}]\neq 0$, being
$Skw({\frak L},\xi)_{[\alpha]}$ a weight module over $Sym({\frak
L},\xi)_{[\alpha]}$.

\section{The simple components}

In this section we are showing that, under certain conditions, the
decomposition of $V$ given in Theorem \ref{lema_final} can be
given by means of  the family of its minimal submodules, each one
being a simple (weight) submodule.

\begin{lemma}\label{lema2}
Suppose ${\mathcal Z}(V)=0$. If $W$ is a  submodule of $V$ such
that $W \subset V_0$, then $W=\{0\}$.
\end{lemma}
\begin{proof}
 On the one hand  $LW \subset W \subset
V_0$, and on the other hand $$LW = (H \oplus (\bigoplus_{\alpha
\in \Lambda}{L_{\alpha}}))W \subset \bigoplus\limits_{\beta \in
\Lambda \cap \mathcal{P}}{V_{\beta}}.$$ Hence, $LW \subset V_0
\cap (\bigoplus\limits_{\beta \in \Lambda \cap
\mathcal{P}}{V_{\beta}})=0$. From here $W=\{0\}$.
\end{proof}

We recall the following definition (see \cite{Benkart, Britten}).

\begin{Definition}
A weight module $V$  it is  said {\rm completely pointed} if
$\dim{V}_{\gamma}=1$ for any $\gamma \in \mathcal{P}$.
\end{Definition}

Let us note that completely pointed weight modules  appears in a
natural way in several contexts.  See for instance \cite{Kp},
\cite{SZ} and \cite{Z} for the cases over Virasoro, generalized
Virasolo  and Witt algebras respectively.

Given any submodule $W$ of $V$, it is well known that  any
submodule of a weight module is again a weight module. So we have
$$W= (W \cap V_0) \oplus (\bigoplus\limits_{\gamma \in \mathcal{P}}
(W \cap V_{\gamma})).$$ Observe that if $V$ is completely pointed
then we can write
\begin{equation}\label{form_subm}
 \hbox{$W = (W \cap V_0) \oplus
(\bigoplus\limits_{\gamma \in \mathcal{P}^W} V_{\gamma})$  where
$\mathcal{P}^W := \{\gamma \in \mathcal{P} : W \cap V_{\gamma}
\neq 0\}$}.
\end{equation}

Let us also introduce the concept of weight-multiplicativity  in
the framework of weight modules over spit Lie algebras, in a
similar way to the analogous one for split Lie algebras, split Lie
triple systems and split twisted inner derivation triple systems,
(see  \cite{Yoalg, Yotriple1, Twisted} for these notions and
examples).
\begin{Definition}
We say that a weight module $V$ over a split Lie algebra $L$ is
{\rm weight-multiplicative} if given $\alpha \in \Lambda$ and
$\gamma \in \mathcal{P}$ such that $\alpha+\gamma \in
\mathcal{P}$, then $L_{\alpha}V_{\gamma} \neq 0.$
\end{Definition}

Here we note that if $V$ satisfies $V_0 = \sum\limits_{\beta \in
\Lambda \cap \mathcal{P}}{L_{-\beta}V_{\beta}}$ we will understand
the weight-multiplicativity of $V$ by supposing also that if
$L_{\beta}(L_{-\beta}V_{\beta}) \neq 0$ then
$L_{-\beta}(L_{\beta}V_{-\beta}) \neq 0$.

\begin{theorem}\label{last}
Let $V$ be a completely pointed weight module,
weight-multiplicative and with ${\mathcal Z}(V)=0$ over a split
Lie algebra $L$.  If $V$ has  all its nonzero  weights connected
and  $V_0 = \sum\limits_{\alpha \in \Lambda \cap
\mathcal{P}}{L_{-\alpha}V_{\alpha}}$ then either $V$ is simple or
$V=W \oplus W^{\prime}$ with $W$ and  $W^{\prime}$ simple (weight)
submodules of $V$.
\end{theorem}
\begin{proof}
Consider $W$ a nonzero submodule of $V$. By Lemma \ref{lema2} and
equation (\ref{form_subm}) we can write $W = (W \cap V_0) \oplus
(\bigoplus\limits_{\gamma \in \mathcal{P}^W} V_{\gamma})$ with
$\mathcal{P}^W \subset \mathcal{P} $ and $\mathcal{P}^W \neq
\emptyset$. Hence, we can take $\gamma_0 \in \mathcal{P}^W$ being
so
\begin{equation}\label{eq12}
0 \neq V_{\gamma_0} \subset W.
\end{equation}
  For $\delta \in
\mathcal{P} \setminus \{\pm \gamma_0\}$ we have $\gamma_0$ is
connected to $\delta.$ Therefore, there exists a connection
$\{\gamma_0, \alpha_1,...,\alpha_n\}$ from $\gamma_0$ to $\delta$
such that
$$\gamma_0,
\gamma_0+\alpha_1,...,\gamma_0+\alpha_1+\cdots+\alpha_{n-1} \in
\mathcal{P}$$ and $$\gamma_0+\alpha_1+\cdots+\alpha_n=
\epsilon_{\tiny{\delta}} \delta,$$ for some
$\epsilon_{\tiny{\delta}} \in \pm 1$. Consider $\gamma_0$,
$\alpha_1$ and $\gamma_0+\alpha_1$. The weight-multiplicativity of
$W$ gives us $0\neq L_{\alpha_1}V_{\gamma_0}$. Hence, the
completely pointed character  of $W$ implies $0\neq
L_{\alpha_1}V_{\gamma_0}=V_{\gamma_0+\alpha_1}$ and equation
(\ref{eq12}) let us conclude $$0\neq V_{\gamma_0+\alpha_1} \subset
W.$$ We can argue in a similar way from $\gamma_0+\alpha_1$,
$\alpha_2$ and $\gamma_0+\alpha_1+ \alpha_2$ to get
$$0\neq
V_{\gamma_0+\alpha_1+\alpha_2} \subset W.$$ Following this process
we obtain that
$$0\neq
V_{\gamma_0+\alpha_1+\alpha_2+\cdots+\alpha_n} \subset W$$ and so
\begin{equation}\label{eq30}
 \hbox{$0\neq V_{\epsilon_{\tiny{\delta}} \delta} \subset W$ for any $\delta \in \mathcal{P}$ and some
 $\epsilon_{\tiny{\delta}}
 \in \pm 1$}.
 \end{equation}

 Let us distinguish two possibilities. In the first one $-\gamma_0 \in \mathcal{P}^W$. Then $$\{-\gamma_0,
-\alpha_1,...,-\alpha_n\}$$ is a connection from $-\gamma_0$ to
$\delta$ satisfying
$$-\gamma_0-\alpha_1-\cdots-\alpha_n= -\epsilon_{\tiny{\delta}}
\delta.$$ By arguing as above we get
$$0\neq V_{-\epsilon_{\tiny{\delta}} \delta} \subset W$$
and so  $ \mathcal{P}^W=\mathcal{P}.$ From here, we also have $V_0
= \sum\limits_{\alpha \in \Lambda \cap
\mathcal{P}}{L_{-\alpha}V_{\alpha}} \subset W$ and so $W = V$.
Hence $V$ is simple.

 In the second possibility,
there is not any $\gamma_0 \in \mathcal{P}^W$ such that $-\gamma_0
\in \mathcal{P}^W$. By equation (\ref{eq30}) we can write
\begin{equation}\label{eq50}
\mathcal{P} = \mathcal{P}^W \dot{\cup} -\mathcal{P}^{ W}
\end{equation}
 where
$-\mathcal{P}^{ W} := \{-\gamma: \gamma \in \mathcal{P}^W\}.$ Let
us denote by
$$W^{\prime}:=(\sum\limits_{\beta \in \Lambda \cap
-\mathcal{P}^W}L_{-\beta}V_{\beta})  \oplus
(\bigoplus\limits_{\delta \in -\mathcal{P}^{W}}{V_{\delta}}).$$ We
have that $W^{\prime}$ is a submodule of $V$. Indeed,
$$LW^{\prime}= (H \oplus (\bigoplus\limits_{\alpha \in \Lambda}L_{\alpha}))
((\sum\limits_{\beta \in \Lambda \cap
-\mathcal{P}^W}L_{-\beta}V_{\beta})  \oplus
(\bigoplus\limits_{\delta \in -\mathcal{P}^{W}}{V_{\delta}}))
\subset$$
\begin{equation}\label{eq40}
(\bigoplus\limits_{\delta \in
-\mathcal{P}^{W}}{V_{\delta}})+(\bigoplus\limits_{\alpha \in
\Lambda}L_{\alpha})(\sum\limits_{\beta \in \Lambda \cap
-\mathcal{P}^W}L_{-\beta}V_{\beta})+(\bigoplus\limits_{\alpha \in
\Lambda}L_{\alpha})(\bigoplus\limits_{\delta \in
-\mathcal{P}^{W}}{V_{\delta}})).
\end{equation}
Consider the second summand in equation (\ref{eq40}). If some
$L_{\alpha}(L_{-\beta}V_{\beta}) \neq 0$ we have that in case
$\alpha =  \beta$, clearly
$L_{\alpha}(L_{-\beta}V_{\beta})=L_{\beta}(L_{-\beta}V_{\beta})\subset
V_{ \beta}\subset W^{\prime},$  and in case $\alpha = -\beta$,
since the facts $W$ is a submodule of $V$ and $\beta \notin
{\mathcal P}^W$ imply $L_{\beta}(L_{\beta}V_{-\beta})=0$, we get
by weight-multiplicativity $L_{-\beta}(L_{-\beta}V_{\beta})=0$.
Suppose then $\alpha \neq \pm \beta$.
 Then either
$[L_{\alpha}L_{-\beta}]V_{\beta}\neq 0$ or
$L_{-\beta}(L_{\alpha}V_{\beta})\neq 0$ and, by the completely
pointed character  of $V$, either
$[L_{\alpha}L_{-\beta}]V_{\beta}=V_{\alpha}$ or
$L_{-\beta}(L_{\alpha}V_{\beta})=V_{\alpha}.$ In the first case,
since $-\beta \in \mathcal{P}^{W}$, we have by
weight-multiplicativity that
$L_{-\alpha+\beta}V_{-\beta}=V_{-\alpha} \subset W.$ That is,
$-\alpha \in \mathcal{P}^{W}$. From here $\alpha \in
-\mathcal{P}^{W}$ and then $V_{\alpha} \subset W^{\prime}$. In the
second case we have in a similar way that
$L_{\beta}(L_{-\alpha}V_{-\beta})=V_{-\alpha}\subset W$ and so
$V_{\alpha} \subset W^{\prime}$. Therefore
$$(\bigoplus\limits_{\alpha \in
\Lambda}L_{\alpha})(\sum\limits_{\beta \in \Lambda \cap
-\mathcal{P}^W}L_{-\beta}V_{\beta}) \subset W^{\prime}.$$ Finally,
if we consider the third summand in equation (\ref{eq40}) and some
$L_{\alpha}V_{\delta}\neq 0$, we have that in case $\alpha =
-\delta$, clearly $L_{\alpha}V_{\delta}=L_{-\delta}V_{\delta}
\subset W^{\prime}$. Suppose then  $\alpha \neq -\delta$ and so
$L_{\alpha}V_{\delta}=V_{\alpha + \delta}$. Since $-\delta \in
\mathcal{P}^{W}$, the weight-multiplicativity of $V$ gives us
$L_{-\alpha}V_{-\delta}=V_{-\alpha - \delta}\subset W$. Hence
$\alpha +\delta \in -\mathcal{P}^{W}$ and then $V_{\alpha +
\delta} \subset W^{\prime}$. Consequently
$(\bigoplus\limits_{\alpha \in
\Lambda}L_{\alpha})(\bigoplus\limits_{\delta \in
-\mathcal{P}^{W}}{V_{\delta}}) \subset W^{\prime}$ and
$W^{\prime}$ is a submodule of $V$.

Now observe that we can write the  direct sum
$$(\sum\limits_{\alpha \in \Lambda \cap
\mathcal{P}^W}L_{-\alpha}V_{\alpha})  \oplus (\sum\limits_{\beta
\in \Lambda \cap -\mathcal{P}^W}L_{-\beta}V_{\beta}).$$ In fact,
if there exists $0\neq v_0 \in (\sum\limits_{\alpha \in \Lambda
\cap \mathcal{P}^W}L_{-\alpha}V_{\alpha})  \cap
(\sum\limits_{\beta \in \Lambda \cap
-\mathcal{P}^W}L_{-\beta}V_{\beta})$,  and taking into account
${\mathcal Z}(V)=0$, $L$ is split and $HV_0=0$, there exists
$0\neq e_{\alpha^{\prime}} \in L_{\alpha^{\prime}}$,
$\alpha^{\prime} \in \Lambda$, such that $e_{\alpha^{\prime}} v_0
\neq 0$, being then $V_{\alpha^{\prime} } \in W \cap W^{\prime}$,
a contradiction. Hence  $v_0=0$ and the   sum is direct.

Taking into account the above observation, the fact $V_0 =
\sum\limits_{\alpha \in \Lambda \cap
\mathcal{P}}{L_{-\alpha}V_{\alpha}}$ and equation (\ref{eq50}) we
have $$V=W \oplus W^{\prime}.$$

 Finally, we note that can argue
with   $W$ and $W^{\prime}$ as we did at the first of the proof
with $V$
to conclude that $W$ and $W^{\prime}$ are simple submodules, which
completes the proof of the theorem.
\end{proof}

\begin{theorem}\label{finalisimo}
Let $V$ be a completely pointed weight module,
weight-multiplicative and with  $LV=V$, ${\mathcal Z}(V)=0$ over a
split Lie algebra $L$.  Then $L =\bigoplus\limits_{i\in I} I_{i}$
 with any $I_{i}$ a nonzero ideal of $L$
satisfying $[I_{i},I_{j}]=0$ if $i \neq j$, and  $V =
\bigoplus\limits_{k \in K} V_{k}$ is the direct sum of the family
of its minimal  submodules, each one being a simple weight
submodule  of $V$ in such a way that for any $k \in K$ there
exists a unique $i \in I$ such that $I_iV_k \neq 0.$ Furthermore
$V_k$ is a weight module over $I_i$.
\end{theorem}

\begin{proof}
By Theorem \ref{lema_final},
\begin{equation}\label{60}
 V = \bigoplus\limits_{j \in J}V_j
 \end{equation}
is the direct sum of the weight submodules   $$V_j =
V_{[\gamma]}=(\sum\limits_{\alpha \in \Lambda \cap
\mathcal{P}_{\gamma}}L_{-\alpha}V_{\alpha}) \oplus
(\bigoplus\limits_{\delta \in \mathcal{P}_{\gamma}}V_{\delta}),$$
$[\gamma] \in \mathcal{P}/\sim$, having any $V_{[\gamma]}$ its
weight system, ${\mathcal P}_{\gamma}$, with all of its weights
connected. Furthermore $L =\bigoplus\limits_{i\in I} I_{i}$
 with any $I_{i}$ an ideal of $L$
satisfying that $[I_{i},I_{k}]=0$ if $i \neq k$,  that for any $j
\in J$ there exists a unique $i \in I$ such that $I_iV_j \neq 0$
and that   $V_j$ is a weight module over $I_i$. We also have that
any of the $V_{[\gamma]}$ is weight-multiplicative as consequence
of the weight-multiplicativity of $V$. Clearly $V_{[\gamma]}$ is
completely pointed, and finally ${\mathcal Z}(V_{[\gamma]})=0$ as
consequence of ${\mathcal Z}(V)=0$. We can apply Theorem
\ref{last} to any $V_{[\gamma]}$ so as to conclude that either
$V_{[\gamma]}$ is simple or $V_{[\gamma]}=W_{[\gamma]} \oplus
W_{[\gamma]}^{\prime}$ with $W_{[\gamma]}$ and
$W_{[\gamma]}^{\prime}$ simple submodules of $V$. From here, it is
clear that by writing $V_j=W_j \oplus W_j^{\prime}$ in equation
(\ref{60}) if $V_j$ is not simple,  we get that the resulting
decomposition satisfies the assertions of the theorem.
 \end{proof}

\section{Some applications and final remarks}

 We recall that split Lie algebras plays an important role in
mathematical physics. As it is explained in \cite{d4}, the split
decomposition of a  Lie algebra is one of the most interesting
fine gradings of a Lie algebra, with a heavy influence in the
field of particle physics via the usual identification of
observables with generators in a splitting Cartan subalgebra, and
particles living comfortably in the root spaces. Also, this split
decomposition appears as a common feature in many attempts of
describing  strong interactions in nature. Other independent role
comes from the theory of contraction of Lie algebras. As claimed
in \cite{g2+}, contractions are important in physics because they
explain in terms of Lie algebras why some theories arise as a
limit of more exact theories. Consider a split Lie algebra $$L= H
\oplus (\bigoplus\limits_{\alpha \in \Lambda}{L}_{\alpha}).$$ We
have that split Lie algebras are examples of weight modules over
themselves, where
 $ \mathcal{P}=\Lambda$ and  $V_{\alpha} = L_{\alpha}$ for $\alpha \in \mathcal{P} \cup
 \{0\}$. From here, Definition \ref{def nueva conec} and  Theorem \ref{finalisimo} let us assert that if
 $L$ is completely pointed, weight-multiplicative and perfect,
 then $L$ is the  direct sum of the family
of its minimal  ideals, $L =\bigoplus\limits_{[\alpha] \in
\Lambda/\approx} I_{[\alpha]}$,  each one being a simple split Lie
algebra satisfying $[I_{[\alpha]},I_{[\beta]}]=0$ if $[\alpha]
\neq [\beta]$.

Let us now center our attention on split Lie superalgebras
$L=L^{\bar
 0} \oplus L^{\bar 1}$, see
\cite{Super}.  Lie superalgebras plays an important role in
theoretical physics, specially in conformal field theory and
supersymmetries (see \cite{JMP1, JMP2, JMP3} for recent
references). The notion of supersymmetry reflects the known
symmetry between bosons and fermions, being the mathematical
structure formalizing this idea the one of supergroup, or ${\hu
Z}_2$-graded Lie group. As mentioned in \cite{Cand11}, its job is
that of modelling continuous supersymmetry transformations between
bosons and fermions. As Lie algebras consist of generators of Lie
groups, the infinitesimal Lie group elements tangent to the
identity, so ${\hu Z}_2$--graded Lie algebras, otherwise known as
Lie superalgebras, consist of generators of (or infinitesimal)
supersymmetry transformations. We also refers to    \cite{bena}
and \cite{Gie} for more interesting applications of Lie
superalgebras. We have that
 $L=L^{\bar
 0} \oplus L^{\bar 1}$ is a weight module over the split Lie
 algebra $L^{\bar
 0}$. So Theorem \ref{finalisimo} let us also assert that if
 $L$ is completely pointed, weight-multiplicative and $L^{\bar
 0}$ is perfect,
 then $L$ is the  direct sum of the family
of its minimal (graded) ideals,  each one being a simple split Lie
superalgebra.

Consider now a Lie triple system $T$ with triple product $[\cdot,
\cdot, \cdot]$. As it is pointed out in \cite{JPA-1}, Lie triple
systems are well related to the theory of quantum mechanics with
$PT$-symmetric Hamiltonians and Krein space-related models in
general, by identifying this underlying structure in the
recognizing of $PT$-like involutory structures in physical models.
If we consider a split Lie triple system $T$, see
\cite{Yotriple1}, it is well known that the even part ${\frak L}$
of its standard embedding  is a split Lie algebra in such a way
that the natural action $${\frak L} \times T \to T$$ makes $T$ a
weight module over the split Lie algebra ${\frak L}$. From here
the results in the present paper apply and we can also conclude
that if
 $T$ is completely pointed, weight-multiplicative and ${\frak L}$ is perfect,
 then $T$ is the  direct sum of the family
of its minimal ternary ideals, $T =\bigoplus\limits_{[\alpha] \in
\Lambda/\approx} I_{[\alpha]}$,  each one being a simple split Lie
triple system  satisfying $[I_{[\alpha]},T,I_{[\beta]}]=0$ if
$[\alpha] \neq [\beta]$.

Finally, we would like to centering on Jordan triple systems $(J,
\{\cdot, \cdot, \cdot \})$. Jordan triple systems also plays a
sensible role in the study of four-dimensional superstring or
heterotic string theories, and in the study of orthosymmetric
ortholattices which are fundamental to approximate to Hilbert
lattices (see \cite{JPA-2, JPA-3}). Split Jordan triple systems
are a particular case of split twisted inner derivation triple
systems, see \cite{Twisted}. The even part ${\frak L}$ of the
standard embedding
 of $J$ is also a split Lie algebra in such a way that
the natural action ${\frak L} \times J \to J$ also makes $J$ a
weight module over the split Lie algebra ${\frak L}$. From here,
we have as in the previous example that  the results in the
present paper apply and   we  get that if
 $J$ is completely pointed, weight-multiplicative and ${\frak L}$ is perfect,
 then $J$ is the  direct sum of the family
of its minimal ternary ideals, $J =\bigoplus\limits_{[\alpha] \in
\Lambda/\approx} I_{[\alpha]}$,  each one being a simple split
Jordan triple system  satisfying
$\{I_{[\alpha]},T,I_{[\beta]}\}+\{I_{[\alpha]},I_{[\beta]},T\}+\{T,I_{[\alpha]},I_{[\beta]}\}=0$
if $[\alpha] \neq [\beta]$.

\medskip

\medskip

\begin{remark}
An interesting open question is to describe simple weight modules
(of arbitrary dimension) $V$ over split Lie algebras. Motivated by
the results in \cite{simplealgebra} and \cite{simpletriple} on
simple split Lie algebras and simple split Lie triple systems
respectively, we conjecture $V$ could be described as a direct
limit of well-known finite dimensional weight modules. Also an
interesting question is that of studying weight modules  over
split Lie algebras with  non-necessarily symmetric weight and root
systems. These will be the topics of a future research of the
authors.
\end{remark}

\medskip



\end{document}